\documentclass[11pt,a4paper,headinclude,footinclude,fleqn]{scrartcl}
\usepackage[T1]{fontenc}                   
\usepackage[utf8]{inputenc}                 
\usepackage[english]{babel}       
\usepackage{graphicx}                      %
\usepackage[font=small]{quoting}            %
\usepackage{caption}                  %
\usepackage[nochapters,beramono,eulermath,%
            pdfspacing,
            listings,
            ]{classicthesis}                %
\usepackage{arsclassica}                    
\usepackage[top=.8in,bottom=1in,left=1in,right=1in]{geometry}
\usepackage[english]{babel}
\usepackage{verbatim}
\usepackage{color}
\usepackage{picture}
\usepackage{amsmath,amsthm,amsfonts,amssymb}
\usepackage{comment}
\usepackage{enumitem}
\usepackage{mathtools}

\usepackage{titling} 
\usepackage[hyperpageref]{backref}

\newtheorem{thm}{Theorem}[section]

\newtheorem{prop}[thm]{Proposition}

\newtheorem{defn}{Definition}
\theoremstyle{definition}
\newtheorem{rem}[thm]{Remark}
\theoremstyle{remark}

\newcommand{\ds}{\displaystyle}

\newcommand{\R}{\mathbb{R}}
\newcommand{\N}{\mathbb{N}}

\newcommand{\Rn}{\mathbb R^n}

\newcommand{\de}{\partial}

\DeclareMathOperator{\dive}{div}
\DeclareMathOperator{\Qp}{\mathcal Q_{p}}

\DeclareMathOperator{\esssup}{ess\,sup}
\DeclareMathOperator{\essinf}{ess\,inf}
\DeclareMathOperator{\spt}{spt}

\patchcmd{\abstract}{\scshape\abstractname}{\textbf{\abstractname}}{}{}
\makeatletter 
\def\@makefnmark{} 
\makeatother 

\title{On the second Dirichlet eigenvalue \\ of some nonlinear anisotropic elliptic operators}
\author{Francesco Della Pietra, %
Nunzia Gavitone, %
Gianpaolo Piscitelli \\[.3cm]%
{\em\scriptsize Universit\`a degli studi di Napoli Federico II, Dipartimento di Matematica e Applicazioni ``R. Caccioppoli''} \\ 
{\em\scriptsize Via Cintia, Monte S. Angelo - 80126 Napoli, Italia.}\thanks{Email: f.dellapietra@unina.it, nunzia.gavitone@unina.it, gianpaolo.piscitelli@unina.it
}
}

\begin{document}

\maketitle

\begin{abstract}
Let $\Omega$ be a bounded open set of $\R^{n}$, $n\ge 2$. In this paper we mainly study some properties of the second Dirichlet eigenvalue $\lambda_{2}(p,\Omega)$ of the anisotropic $p$-Laplacian
\[
-\mathcal Q_{p}u:=-\dive \left(F^{p-1}(\nabla u)F_\xi (\nabla u)\right),
\]
where $F$ is a suitable smooth norm of $\R^{n}$ and $p\in]1,+\infty[$. We provide a lower bound of $\lambda_{2}(p,\Omega)$ among bounded open sets of given measure, showing the validity of a Hong-Krahn-Szego type inequality. 
Furthermore, we investigate the limit problem as $p\to+\infty$.
\\
 
\end{abstract}

\section{Introduction}

Let $\Omega$ be a bounded open set of $\R^{n}$, $n\ge 2$. The main aim of this paper is to study some properties of the Dirichlet eigenvalues of the anisotropic $p$-Laplacian operator:
\begin{equation}
\label{operat}
-\mathcal Q_{p}u:=-\dive \left(F^{p-1}(\nabla u)\nabla_{\xi}F (\nabla u)
\right),
\end{equation}
where $1<p<+\infty$, and $F$ is a sufficiently smooth norm on $\R^{n}$ (see Section \ref{anintro} for the precise assumptions on $F$), namely the values $\lambda$ such that the problem
\begin{equation}
\label{introaut}
\left\{
\begin{array}{ll}
-\mathcal Q_{p}u=\lambda |u|^{p-2}u  &\text{in}\ \Omega \\
u=0 &\text{on}\ \partial\Omega
\end{array}
\right.
\end{equation}
admits a nontrivial solution in $W_{0}^{1,p}(\Omega)$.

The operator in \eqref{operat} reduces to the $p$-Laplacian when $F$ is the Euclidean norm on $\R^{n}$. For a general norm $F$, $\mathcal Q_{p}$ is anisotropic and can be highly nonlinear. In literature, several papers are devoted to the study of the smallest eigenvalue of \eqref{introaut}, denoted by $\lambda_{1}(p,\Omega)$, in bounded domains (see Section \ref{anintro}),  which has the variational characterization
\[
\lambda_{1}(p,\Omega)=\min_{\varphi\in W_{0}^{1,p}(\Omega)\setminus\{0\}}\frac{\displaystyle\int_{\Omega}F^{p}(\nabla \varphi)\,dx}{\displaystyle\int_{\Omega}|\varphi|^{p}\,dx}.
\]
Let $\Omega$ be a bounded domain. It is known (see \cite{bfk,dpgpota}) that $\lambda_{1}(p,\Omega)$ is simple, the eigenfunctions have constant sign and it is isolated; moreover the only positive eigenfunctions are the first eigenfunctions. Furthermore, the Faber-Krahn inequality holds (see \cite{bfk}):
\[
\lambda_{1}(p,\Omega)\ge \lambda_{1}(p,\mathcal W_{R})
\]
where $\mathcal W_{R}$ is the so-called Wulff shape, that is the ball with respect to the dual norm $F^{o}$ of $F$, having the same measure of $\Omega$ (see Section \ref{anintro}). Many other results are known for $\lambda_{1}(p,\Omega)$. The interested reader may refer, for example, to \cite{bfk,bkj,bgm,dpgmana,kn08,p,wx2}. As matter of fact, also different kind of boundary conditions have been considered as, for example, in the papers \cite{dgans,dgp2} (Neumann case), \cite{dpgpota}  (Robin case).

Among the results contained in the quoted papers, we recall that if $\Omega$ is a bounded domain, it has been proved in \cite{bkj} that
\[
\lim_{p\to \infty}\lambda_{1}(p,\Omega)^{\frac1p}=\frac{1}{\rho_{F}(\Omega)},
\]
where $\rho_{F}(\Omega)$ is the anisotropic inradius of $\Omega$ with respect to the dual norm (see Section \ref{anintro} for its definition), generalizing a well-known result in the Euclidean case contained in \cite{jlm}.
 
Actually, very few results are known for higher eigenvalues in the anisotropic case. In \cite{thfr} the existence of a infinite  sequence of eigenvalues is proved, obtained by means of a $\min-\max$ characterization. Actually, as in the Euclidean case, it is not known if this sequence exhausts all the set of the eigenvalues. Here we will show that the spectrum of $-\mathcal Q_{p}$ is a closed set, that the eigenfunctions are in $C^{1,\alpha}(\Omega)$ and that admit a finite number of nodal domains. We recall the reference \cite{L2}, where many results for the spectrum of the $p$-Laplacian in the Euclidean case have been summarized.
\\

The core of the paper relies in the study of the second eigenvalue $\lambda_{2}(p,\Omega)$, $p\in]1,+\infty[$, in bounded open sets, 
defined as
\begin{equation*}
\lambda_2(p,\Omega):=\begin{cases}
\min\{\lambda > \lambda_1(p,\Omega)\colon \lambda\ \text{is an eigenvalue}\}& \text{if } \lambda_1(p,\Omega) \text{ is simple}\\
\lambda_1(p,\Omega) & \text{otherwise},
\end{cases}
\end{equation*}
and in analyzing its behavior when $p\to\infty$. 

First of all, we show that if $\Omega$ is a domain, then $\lambda_{2}(p,\Omega)$ admits exactly two nodal domains. Moreover, for a bounded open set $\Omega$, we prove a sharp lower bound for $\lambda_{2}$, namely the Hong-Krahn-Szego inequality
\begin{equation}
\label{HKS}
\lambda_2(p,\Omega)\ge  \lambda_2(p,\widetilde{\mathcal W}),
\end{equation}
where $\widetilde{\mathcal W}$ is the union of two disjoint Wulff shapes, each one of measure $\frac{|\Omega|}{2}$.

In the Euclidean case, such inequality is well-known for $p=2$, and it has been recently studied for any $1<p<+\infty$ in \cite{bf2}.

Finally, we address our attention to the behavior of $\lambda_{2}(p,\Omega)$ when $\Omega$ is a bounded open set and $p\to +\infty$. In particular, we show that
\[
\lim_{p\to \infty}\lambda_{2}(p,\Omega)^{\frac1p} = \frac{1}{\rho_{2,F}(\Omega)},
\]
where $\rho_{2,F}(\Omega)$ is the radius of two disjoint Wulff shapes $\mathcal W_{1},\mathcal W_{2}$ such that $\mathcal W_{1}\cup\mathcal W_{2}$ is contained in $\Omega$. Furthermore, the normalized eigenfunctions of $\lambda_{p}(2,\Omega)$ converge to a function $u_{\infty}$ that is a viscosity solution to a suitable fully nonlinear elliptic problem (see Section 5). In the Euclidean case, this kind of result has been proved for bounded domains in \cite{jl}. We consider both the nonconnected case and general norm $F$. In a forthcoming paper we will deal with the limit case $p \to 1$.

As enhanced before, the aim of the paper is twofold: first, to consider the case of a general Finsler norm $F$; second, to extend  also the results known in the case of domains, to the case of nonconnected sets. We structured the paper as follows. In Section $2$ we recall the main definitions as well as some basic fact of convex geometry, and we fix the precise assumptions on $F$. In Section $3$ we state the general eigenvalue problem for $-\mathcal Q_{p}$, recalling some known results, extending them, where it is possible, to the case of nonconnected sets; moreover we provide several properties of the first and of higher eigenvalues and eigenfunctions. In Section $4$ the attention will be focused on the second eigenvalue $\lambda_{2}(p,\Omega)$ of $-\mathcal Q_{p}$ proving, among the other properties of $\lambda_{2}$, the Hong-Krahn-Szego inequality. Finally, in Section $5$ we study the limit case $p\to \infty$. 

\section{Notation and preliminaries}
\label{anintro}
Throughout the paper we will consider a convex even 1-homogeneous function 
\[
\xi\in \R^{n}\mapsto F(\xi)\in [0,+\infty[,
\] 
that is a convex function such that
\begin{equation}
\label{eq:omo}
F(t\xi)=|t|F(\xi), \quad t\in \R,\,\xi \in \R^{n}, 
\end{equation}
 and such that
\begin{equation}
\label{eq:lin}
a|\xi| \le F(\xi),\quad \xi \in \R^{n},
\end{equation}
for some constant $0<a$. 
Under this hypothesis it is easy to see that there exists $b\ge a$ such that
\[
F(\xi)\le b |\xi|,\quad \xi \in \R^{n}.
\]
Moreover, throughout the paper we will assume that 
\begin{equation}
\label{strong}
\nabla^{2}_{\xi}[F^{p}](\xi)\text{ is positive definite in }\R^{n}\setminus\{0\},
\end{equation}
with $1<p<+\infty$. 

The hypothesis \eqref{strong} on $F$ assures that the operator 
\[
\Qp[u]:= \dive \left(\frac{1}{p}\nabla_{\xi}[F^{p}](\nabla u)\right)
\] 
is elliptic, hence there exists a positive constant $\gamma$ such that
\begin{equation*}
\frac1p\sum_{i,j=1}^{n}{\nabla^{2}_{\xi_{i}\xi_{j}}[F^{p}](\eta)
  \xi_i\xi_j}\ge
\gamma |\eta|^{p-2} |\xi|^2, 
\end{equation*}
for some positive constant $\gamma$, for any $\eta \in
\Rn\setminus\{0\}$ and for any $\xi\in \Rn$. 
\begin{rem}
We stress that for $p\ge 2$ the condition  
\begin{equation*}
\nabla^{2}_{\xi}[F^{2}](\xi)\text{ is positive definite in }\R^{n}\setminus\{0\},
\end{equation*}
implies \eqref{strong}.
\end{rem}
The polar function $F^o\colon\R^n \rightarrow [0,+\infty[$ 
of $F$ is defined as
\begin{equation*}
F^o(v)=\sup_{\xi \ne 0} \frac{\langle \xi, v\rangle}{F(\xi)}. 
\end{equation*}
 It is easy to verify that also $F^o$ is a convex function
which satisfies properties \eqref{eq:omo} and
\eqref{eq:lin}. Furthermore, 
\begin{equation*}
F(v)=\sup_{\xi \ne 0} \frac{\langle \xi, v\rangle}{F^o(\xi)}.
\end{equation*}
From the above property it holds that
\begin{equation*}
|\langle \xi, \eta\rangle| \le F(\xi) F^{o}(\eta), \qquad \forall \xi, \eta \in \R^{n}.
\end{equation*}
The set
\[
\mathcal W = \{  \xi \in \R^n \colon F^o(\xi)< 1 \}
\]
is the so-called Wulff shape centered at the origin. We put
$\kappa_n=|\mathcal W|$, where $|\mathcal W|$ denotes the Lebesgue measure
of $\mathcal W$. More generally, we denote with $\mathcal W_r(x_0)$
the set $r\mathcal W+x_0$, that is the Wulff shape centered at $x_0$
with measure $\kappa_nr^n$, and $\mathcal W_r(0)=\mathcal W_r$.

The following properties of $F$ and $F^o$ hold true
(see for example \cite{bp}):
\begin{gather*}
 \langle\nabla_{\xi}F(\xi) , \xi \rangle= F(\xi), \quad  \langle\nabla_{\xi}F^{o} (\xi), \xi \rangle
= F^{o}(\xi),
 \\
 F( \nabla_{\xi}F^o(\xi))=F^o( \nabla_{\xi}F(\xi))=1,\quad \forall \xi \in
\R^n\setminus \{0\}, 
\\
F^o(\xi)  \nabla_{\xi}F(\nabla_{\xi}F^o(\xi) ) = F(\xi) 
\nabla_{\xi}F^o( \nabla_{\xi}F(\xi) ) = \xi\qquad \forall \xi \in
\R^n\setminus \{0\}. 
\end{gather*}

Finally, we will recall the following
\begin{defn}
A domain of $\mathbb R^{n}$ is a connected open set.
\end{defn}

\section{The Dirichlet eigenvalue problem for \texorpdfstring{$-\mathcal Q_{p}$}{TEXT}}
Here we state the eigenvalue problem for $\Qp$. Let $\Omega$ be a bounded open set in $\R^{n}$, $n\ge 2$, $1<p<+\infty$, and consider the problem
\begin{equation}
\label{eigpb}
\left\{
\begin{array}{ll}
-\Qp u=\lambda |u|^{p-2}u & \text{in } \Omega \\
u=0 &\text{on } \partial\Omega.
\end{array}
\right.
\end{equation}

\begin{defn}
We say that $u \in W^{1,p}_0(\Omega)$, $u \ne 0$, is an eigenfunction of \eqref{eigpb}, if
 \begin{equation}
 \label{eig-def}
\int_{\Omega}\langle F^{p-1}(\nabla u) \nabla_{\xi}F(\nabla u),\nabla\varphi \rangle\ dx= \lambda\int_\Omega |u|^{p-2} u \varphi \ dx
 \end{equation}
for all $\varphi \in W^{1,p}_0(\Omega)$. The corresponding real number $\lambda$ is called an eigenvalue of \eqref{eigpb}.
\end{defn}

Obviously, if $u$ is an eigenfunction associated to $\lambda$, then
\begin{equation*}
	\lambda=\frac{\ds\int_\Omega F^p(\nabla u)\ dx}{\ds\int_\Omega |u|^p\ dx}>0.
\end{equation*}

\subsection{{The first eigenvalue}}
Among the eigenvalues of \eqref{eigpb}, the smallest one, denoted here by $\lambda_{1}(p,\Omega)$, has the following well-known variational characterization:
\begin{equation}
\label{rayleigh}
\lambda_{1}(p,\Omega)=\min_{\varphi\in W^{1,p}_{0}(\Omega)\setminus\{0\}} \frac{\ds\int_\Omega F^p(\nabla \varphi)\ dx}{\ds\int_\Omega |\varphi|^p\ dx}.
\end{equation}
In the following theorems its main properties are recalled. 

\begin{thm}
\label{th_prop1}
 If $\Omega$ is a bounded open set in $\R^{n}$, $n\ge 2$, there exists a function $u_{1}\in C^{1,\alpha}(\Omega)\cap C(\overline{\Omega})$ which achieves the minimum in \eqref{rayleigh}, and satisfies the problem \eqref{eigpb} with $\lambda=\lambda_{1}(p,\Omega)$. Moreover, if $\Omega$ is connected, then $\lambda_{1}(p,\Omega)$ is simple, that is the corresponding eigenfunctions are unique up to a multiplicative constant, and the first eigenfunctions have constant sign in $\Omega$. 
 \end{thm}
\begin{proof}
The proof can be immediately adapted from the case of $\Omega$  connected and we refer the reader, for example, to \cite{li0,bfk}.
\end{proof}

\begin{thm}
\label{th_prop2}
 Let $\Omega$ be a bounded open set  in $\R^{n}$, $n\ge 2$. Let  $u \in W_{0}^{1,p}(\Omega)$ be an eigenfunction of \eqref{eigpb} associated to an eigenvalue $\lambda$. If $u$ does not change sign in $\Omega$, then there exists a connected component $\Omega_{0}$ of $\Omega$ such that $\lambda=\lambda_{1}(p,\Omega_{0})$ and $u$ is a first eigenfunction in $\Omega_{0}$. In particular, if $\Omega$ is connected then $\lambda=\lambda_{1}(p,\Omega)$ and a constant sign eigenfunction is a first eigenfunction.
\end{thm}
\begin{proof}
If $\Omega$ is connected, a proof can be found in \cite{li0,dpgpota}. Otherwise, if $u\ge 0$ in $\Omega$ disconnected, by the maximum principle $u$ must be either positive or identically zero in each connected component of $\Omega$. Hence there exists a connected component $\Omega_{0}$ such that $u$ coincides in $\Omega_{0}$ with a positive eigenfunction relative to $\lambda$. By the previous case, $\lambda=\lambda_{1}(p,\Omega_{0})$ and the proof is completed.
\end{proof}
Here we list some other useful and interesting properties that can be proved in a similar way than the Euclidean case.
\begin{prop}
\label{fireig}
 Let $\Omega$ be a bounded open set in $\R^{n}$, $n\ge 2$, the following properties hold.
\begin{enumerate}
\item For $t>0$ it holds $\lambda_{1} (p,\,t\Omega)=t^{-p}\lambda_1 (p,\Omega)$.
\item If $\Omega_1\subseteq\Omega_2\subseteq\Omega$, then $\lambda_1(p,\Omega_1)\ge \lambda_1(p,\Omega_2)$.
\item For all $1< p< s< +\infty$ we have $p[\lambda_1(p,\Omega)]^{1/p}<s[\lambda_1(s,\Omega)]^{1/s}$.
\end{enumerate}
\end{prop}

\begin{proof}
The first two properties are immediate from \eqref{rayleigh}. As regards the third property, the inequality derives from the H\"older inequality, similarly as in \cite{lpota}. Indeed, taking $\phi=|\psi|^{\frac sp-1}\psi$, $\psi\in W_{0}^{1,p}(\Omega)\cap L^{\infty}(\Omega)$, $\psi\ge 0$, we have by \eqref{eq:omo} that
\[
[\lambda_{1}(p,\Omega)]^{\frac1p}\le \frac{s}{p}\left( \frac{\ds\int_{\Omega}|\psi|^{s-p}F^{p}(\nabla \psi)dx}{\ds\int_{\Omega}|\psi|^{s}dx} \right)^{\frac 1p} \le \frac{s}{p}
\left( \frac{\ds\int_{\Omega} F^{s}(\nabla \psi)dx}{\ds\int_{\Omega}|\psi|^{s}dx} \right)^{\frac 1s}
\]
By minimizing with respect to $\psi$, we get the thesis.
\end{proof}
In addition, the Faber-Krahn inequality for $\lambda_{1}(p,\Omega)$ holds. 
\begin{thm}
 Let $\Omega$ be a bounded open set in $\R^{n}$, $n\ge 2$, then
\begin{equation}
\label{faberk}
|\Omega|^{p/N}\lambda_1(p,\Omega)\geq\kappa_N^{p/N}\lambda_1(p,\mathcal W).
\end{equation}
Moreover, equality sign in \eqref{faberk} holds if $\Omega$ is homothetic to the Wulff shape.
\end{thm}
The proof of this inequality, contained in \cite{bfk}, is based on a symmetrization technique introduced in \cite{aflt} (see \cite{etpolya,fvpolya} for the equality cases).

Using the previous result we can prove the following property of $\lambda_{1}(p,\Omega)$.
\begin{prop}
\label{isoprop}
 Let $\Omega$ be a bounded domain in $\R^{n}$, $n\ge 2$. The first eigenvalue of \eqref{eigpb}, $\lambda_1(p,\Omega)$, is isolated.
\end{prop}
\begin{proof}
We argue similarly  as in \cite{L2}. For completeness we give the proof. For convenience we write $\lambda_{1}$ instead of $\lambda_{1}(p,\Omega)$.  Let $\lambda_{k}\ne \lambda_{1}$ a sequence of eigenvalues such that 
\[
\lim_{k \to +\infty}\lambda_{k}=\lambda_{1}
\]
Let $u_{k}$ be a normalized  eigenfunction associated to $\lambda_{k}$
that is,
\begin{equation}
\label{propsucc}
\lambda_{k}=\int_{\Omega}F^{p}(\nabla u_{k})\,dx \quad \text{ and } \quad\int_{\Omega}|u_k|^{p} \,dx=1
\end{equation}
By \eqref{propsucc}, there exists a function $u \in W_{0}^{1,p}(\Omega)$ such that, up to a subsequence
\[
u_{k} \to u \quad \text{ in }L^{p}(\Omega)  \quad \nabla u_{k} \rightharpoonup \nabla u\text{ weakly in }L^{p}(\Omega).
\] 
By the strong convergence of $u_{k}$ in $L^{p}(\Omega)$ and, recalling that $F$ is convex, by weak lower semicontinuity, it follows that
\[
\ds\int_{\Omega}|u|^{p}dx=1\quad\text{and}\quad \int_{\Omega}F^{p}(\nabla u)\, dx \le \lim_{k\to \infty} \lambda_{k}=\lambda_{1}.
\]
Hence, $u$ is a first eigenfunction. On the other hand, being $u_{k}$ not a first eigenfunction, by Theorem \ref{th_prop2} it has to change sign. Hence, the sets $\Omega_{k}^{+}=\{u_{k}>0\}$ and $\Omega_{k}^{-}=\{u_{k}<0\}$ are nonempty and, as a consequence of the Faber-Krahn inequality and of Theorem \ref{th_prop2}, it follows that
\[
\lambda_{k}=  \lambda_{1}(p,\Omega^{+}_{k})\ge\frac{C_{n,F}}{|\Omega^{+}_{k}|^{\frac pn}},\qquad
\lambda_{k}=  \lambda_{1}(p,\Omega^{-}_{k})\ge\frac{C_{n,F}}{|\Omega^{-}_{k}|^{\frac pn}}. 
\]
This implies that both $|\Omega^{+}_{k}|$ and $|\Omega_{k}^{-}|$ cannot vanish as $k\to +\infty$ and finally, that $u_{k}$ converges to a function $u$ which changes sign in $\Omega$. This is in contradiction with the characterization of the first eigenfunctions, and the proof is completed.
\end{proof}

\subsection{Higher eigenvalues}
First of all, we recall the following result (see \cite[Theorem 1.4.1]{thfr} and the references therein), which assures the existence of infinite eigenvalues of $-\mathcal{Q}_{p}$. We use the following notation. Let
$\mathbb S^{n-1}$ be the unit Euclidean sphere in $\R^{n}$, and 
\begin{equation}
\label{emme}
M=\{u\in W_{0}^{1,p}(\Omega)\colon \int_{\Omega}|u|^{p}dx=1 \}.
\end{equation}
Moreover, let $\mathcal C_{n}$ be the class of all odd and continuous mappings from $\mathbb S^{n-1}$ to $M$. Then, for any fixed $f\in \mathcal C_{n}$, we have $f:\omega\in \mathbb S^{n-1}\mapsto f_{\omega}\in M$. 
\begin{prop}
\label{infiniti}
Let $\Omega$ a bounded open set of $\R^{n}$, for any $k \in \N$, the value
\begin{equation*}
\tilde\lambda_{k}(p,\Omega)=\inf_{f\in \mathcal C_{n}}\max_{\omega\in \mathbb S^{n-1}}\int_{\Omega} F^{p}(\nabla f_{\omega})dx
\end{equation*}
is an eigenvalue of $-\mathcal Q_{p}$. Moreover, 
\[
0< \tilde\lambda_{1}(p,\Omega)=\lambda_{1}(p,\Omega)\le \tilde\lambda_{2}(p,\Omega)\le \ldots \le \tilde\lambda_{k}(p,\Omega)\le \tilde\lambda_{k+1}(p,\Omega)\le \ldots,
\]
and
\[
\tilde\lambda_{k}(p,\Omega)\to \infty \text{ as }k\to \infty.
\]
\end{prop}
Hence, we have at least a sequence of eigenvalues of $-\Qp$. Furthermore, the following proposition holds.
\begin{prop}
Let $\Omega$ a bounded open set of $\R^{n}$. The spectrum of $-\Qp$ is a closed set. 
\end{prop}
\begin{proof}
Let $\lambda_k$ be a sequence of eigenvalues converging to $\mu<+\infty$ and let $u_k $ be the corresponding normalized eigenfunctions, that is such that $\|u_k\|_{L^p(\Omega)}=1$. We have to show that $\mu$ is an eigenvalue of $-\Qp$.

We have that
\begin{equation}
\label{limeq}
\int_{\Omega}\langle F^{p-1}(\nabla u_{k}) \nabla_{\xi}F(\nabla u_{k}),\nabla\varphi \rangle dx= \lambda_{k}\int_\Omega |u_{k}|^{p-2} u_{k} \varphi \, dx
\end{equation}
for any test function $\varphi\in W_{0}^{1,p}(\Omega)$.
Since 
\[
\lambda_k=\int_\Omega F^{p}(\nabla u_k)\, dx,
\]
and being $\lambda_{k}$ a convergent sequence,  up to a subsequence we have that there exists a function $u\in W_{0}^{1,p}(\Omega)$ such that $u_{k}\rightarrow u$ strongly in $L^p(\Omega)$ and $\nabla u_{k}\rightharpoonup \nabla u$ weakly in $L^p(\Omega)$. Our aim is to prove that $u$ is an eigenfunction relative to $\lambda$.

Choosing $\varphi=u_{k}-u$ as test function in the equation solved by $u_{k}$, we have
\begin{equation*}
\begin{split}
&\int_\Omega \langle F^{p-1}(\nabla u_{k })  \nabla_\xi F(\nabla u_{k })-F^{p-1}(\nabla u)  \nabla_\xi F (\nabla u),\nabla  (u_{k}-u)\rangle dx\\
&=\lambda_{k }\int_\Omega |u_{k }|^{p-2}u_{k }(u_{k }-u)\ dx -\int_\Omega F^{p-1}(\nabla u)\langle   \nabla_\xi F  (\nabla u) ,\nabla (u_{k }-u)\rangle dx.
\end{split}
\end{equation*}
By the strong convergence of $u_{k }$ and the weak one of $\nabla u_{k }$, the right-hand side of the above identity goes to zero as $k$ diverges. Hence
\begin{equation*}
\lim_{k\rightarrow \infty}\int_\Omega \langle F^{p-1}(\nabla  u_{k })   \nabla_\xi F (\nabla  u_{k })-F^{p-1}(\nabla  u)  \nabla_\xi F  (\nabla  u), \nabla  (u_{k }-u)\rangle dx=0
\end{equation*}
By nowadays standard arguments, this limit implies the strong convergence of the gradient, hence we can pass to the limit under the integral sign in \eqref{limeq} to obtain
\begin{equation*}
\int_\Omega \langle F^{p-1}(\nabla u) F_\xi(\nabla u),\nabla \varphi\rangle\, dx=\lambda\int_\Omega |u|^{p-2}u\varphi\, dx
\end{equation*}
This shows that $\lambda$ is an eigenvalue and the proof is completed.
\end{proof}
Finally, we list some properties of the eigenfunctions, well-known in the Euclidean case (see for example \cite{L2,aft98}). Recall that a nodal domain of an eigenfunction $u$ is a connected component of $\{u>0\}$ or $\{u<0\}$.
\begin{prop}Let $p>1$, and let $\Omega$ be a bounded open set in $\R^{n}$. Then the following facts hold.
\begin{enumerate} 
\item[(i)]  Any eigenfunction of $-\Qp$ has only a finite number of nodal domains.
\item[(ii)] Let $\lambda$ be an eigenvalue of $-\Qp$, and $u$ be a corresponding eigenfunction. The following estimate holds:
 \begin{equation}\label{L-infty}
 \|u\|_{L^\infty(\Omega)} \le C_{n,p,F}\lambda^\frac{n}{p}\|u\|_{L^{1}(\Omega)}
 \end{equation} 
 where $C_{n,p,F}$ is a constant depending only on $n$, $p$ and $F$.
 \item[(iii)] All the eigenfunctions of \eqref{operat} are in $C^{1,\alpha}(\Omega)$, for some $\alpha\in(0,1)$.
\end{enumerate} 
\end{prop}

\begin{proof}
Let $\lambda$ be an eigenvalue of $-\Qp$, and $u$ a corresponding eigenfunction.

In order to prove $(i)$, let us denote by $\Omega^{+}_{j}$ a connected component of the set $\Omega^+:=\{u>0\}$.
Being $\lambda=\lambda_{1}(\Omega^{+}_{j})$, then by \eqref{faberk} 
\[
|\Omega^{+}_{j}| \ge C_{n,p,F} \lambda^{-\frac np}. 
\]
Then, the thesis follows observing that
\[
|\Omega|\ge \sum_{j}|\Omega^{+}_{j}|\ge C_{n,p,F}\lambda^{-\frac np}\sum_{j}1.
\]

In order to prove $(ii)$, let $k>0$, and choose $\varphi(x)=\max\{u(x) -k, 0\}$ as test function in \eqref{eig-def}. Then
\begin{equation}
\label{lindqvist2.7}
\int_{A_k}F^p(\nabla u)\ dx = \lambda \int_{A_k} |u|^{p-2} u (u-k) \ dx
\end{equation}
where $A_k=\{ x \in\Omega \colon u(x)>k\}$. Being $k |A_k|\leq ||u||_{L^{1}(\Omega)}$, then $|A_k|\rightarrow 0$ as $k\rightarrow\infty$. By the inequality $a^{p-1}\le 2^{p-1}(a-k)^{p-1}+2^{p-1}k^{p-1}$, we have
\begin{equation}
\label{lindqvist2.8}
\int_{A_k} |u|^{p-2} u (u-k) \ dx\le2^{p-1}\int_{A_k}(u-k)^{p}\ dx+2^{p-1}k^{p-1}\int_{A_k}(u-k)\ dx.
\end{equation}
By Poincar\'e inequality and property \eqref{eq:lin}, then \eqref{lindqvist2.7} and \eqref{lindqvist2.8} give that
\[(1-\lambda C_{n,p,F}|A_k|^{p/n})\int_{A_k}(u-k)^p\ dx\le\lambda |A_k|^{p/n} C_{n,p,F}k^{p-1}\int_{A_k}(u-k) \ dx.\]
By choosing $k$ sufficiently large, the H\"older inequality implies
\[
\int_{A_k}(u-k) \ dx \le \tilde C_{n,p,F}\lambda^\frac{1}{p-1}k|A_k|^{1+\frac{p}{n(p-1)}}.
\]  
This estimate allows to apply \cite[Lemma 5.1, p. 71]{LU} in order to get the boundedness of $\esssup u$. Similar argument gives that $\essinf u$ is bounded.

Since \eqref{L-infty} holds, by standard elliptic regularity theory (see e.g. \cite{LU}) the eigenfunction is $C^{1,\alpha}(\Omega)$.
\end{proof}

\section{The second Dirichlet eigenvalue of \texorpdfstring{$-\mathcal Q_{p}$}{TEXT} }

If $\Omega$ is a bounded domain, Proposition \ref{isoprop} assures that the first eigenvalue $\lambda_1(p,\Omega)$ of \eqref{operat} is isolated. This suggests  the following definition.
\begin{defn}
Let $\Omega$ be a bounded open set of $\R^{n}$. Then the second eigenvalue of $-\Qp$ is
\begin{equation*}
\lambda_2(p,\Omega):=\begin{cases}
\min\{\lambda > \lambda_1(p,\Omega)\colon \lambda\ \text{is an eigenvalue}\}& \text{if } \lambda_1(p,\Omega) \text{ is simple}\\
\lambda_1(p,\Omega) & \text{otherwise}.
\end{cases}
\end{equation*}
\end{defn}
\begin{rem}
If $\Omega $ is connected, by theorems \ref{th_prop1} and \ref{th_prop2} we deduce the following characterization of the second eigenvalue:
\begin{equation}
\label{char2}
\lambda_2(p,\Omega)=\min\{\lambda : \lambda \text{ admits a sign-changing eigenfunction}\}.
\end{equation}
\end{rem}
We point out  that in \cite{thfr} it is proved that in a bounded open set it holds
\begin{equation}
\label{caratt2}
\lambda_{2}(p,\Omega)=\tilde\lambda_{2}(p,\Omega)=\inf_{\gamma\in\Gamma_\Omega(u_1,-u_1)}\max_{u\in\gamma([0,1])}\int_\Omega F^p(\nabla u(x)) \ dx
\end{equation}
 where $\tilde\lambda_{2}(p,\Omega)$ is given  in Proposition \ref{infiniti}, and  
 \begin{equation*}
 \Gamma_\Omega(u,v)=\left\{\gamma:[0,1]\rightarrow M \colon \gamma \text{ is continuous and } \gamma(0)=u,\ \gamma(1)=v\right\},
 \end{equation*}
 with $M$ as in \eqref{emme}. As immediate consequence of \eqref{caratt2} we get  
\begin{prop}
\label{monotone}
If $\Omega_1\subseteq\Omega_2\subseteq\Omega$, then $\lambda_2(p,\Omega_{1})\geq\lambda_2(p,\Omega_{2})$.
\end{prop}

By adapting the method contained in \cite{CDG1,CDG2}, it is possible to prove the following result.

\begin{prop}
Let $\Omega$ be a bounded domain in $\R^{n}$.
The eigenfunctions associated to $\lambda_2(p,\Omega)$ admit exactly two nodal domains.
\end{prop}

\begin{proof}
We will proceed as in the proof of \cite[Th. 2.1]{CDG2}. In such a case, $\lambda_{2}(p,\Omega)$ is characterized as in \eqref{char2}. 
Then any eigenfunction $u_{2}$ has to change sign, and it admits at least two nodal domains $\Omega_1\subset\Omega^{+}$ and $\Omega_2\subset\Omega^{-}$. Let us assume, by contradiction, the existence of a third nodal domain $\Omega_3$ and let us suppose, without loss of generality, that $\Omega_{3}\subset\Omega^{+}$. 

\noindent \textbf{Claim.} There exists a connected open set $\widetilde{\Omega}_{2}$, with $\Omega_2\subset\widetilde{\Omega}_{2}\subset\Omega$ such that $\widetilde{\Omega}_{2}\cap \Omega_1=\emptyset$ or $\widetilde{\Omega}_{2}\cap \Omega_3=\emptyset$. 

The proof of the claim follows line by line as in \cite[Th. 2.1]{CDG2}. One of the main tool is the Hopf maximum principle, that for the operator $-\Qp$ is proved for example in \cite[Th. 2.1]{CT}. 

Now, without loss of generality, we assume that $\widetilde{\Omega}_{2}$ is disjoint of $\Omega_1$ and from this fact a contradiction is derived.
 
By the fact that $u_2$ does not change sign on the nodal domains and by Proposition \ref{monotone}, we have that $\lambda_1(p,\Omega_1)=\lambda_2(p,\Omega)$ and that $\lambda_1(p,\widetilde{\Omega}_2)<\lambda_1(p,\Omega_2)=\lambda_2(p,\Omega)$. Now, we may construct the disjoint sets $\widetilde{\widetilde{\Omega}}_2$ and $\widetilde{\Omega}_1$   such that $\Omega_2\subset\widetilde{\widetilde{\Omega}}_2 \subset\widetilde{\Omega}_2$ and $\Omega_1\subset\widetilde{\Omega}_1$, in order to have
\[
 \lambda_{1}(p,\widetilde{\Omega}_{1})<\lambda_{2}(p,\Omega),\qquad\lambda_{1}(p,\widetilde{\widetilde{\Omega}}_2)<\lambda_{2}(p,\Omega).
\]

Now let $v_1$ and $v_2$ be the extension by zero outside $\widetilde{\Omega}_1$ and $\widetilde{\widetilde{\Omega}}_2$, respectively, of the positive normalized eigenfunctions associated to $\lambda_1(p,\widetilde{\Omega}_1)$ and $\lambda(p,\widetilde{\widetilde{\Omega}}_2)$. Hence we easily verify that the function $v=v_1-v_2$ belongs to $W^{1,p}_0(\Omega)$, it changes sign and satisfies
\begin{equation*}
\frac{\int_\Omega F^p(\nabla v_+)\ dx}{\int_\Omega  v_+^p\ dx}<\lambda_2(p,\Omega),\quad \frac{\int_\Omega F^p(\nabla v_-)\ dx}{\int_\Omega  v_-^p\ dx}<\lambda_2(p,\Omega).
\end{equation*}

The final aim is to construct a path $\gamma([0,1])$ such 
\[
\max_{u\in\gamma([0,1])}\int_\Omega F^p(\nabla u(x))\, dx<
\lambda_2(p,\Omega),
\]
obtaining a contradiction from \eqref{caratt2}. The construction of this path follows adapting the method contained in \cite{CDG1,CDG2}.

\end{proof}

\begin{rem} 
\label{Wulff}
In order to better understand the behavior of $\lambda_{1}(p,\Omega)$ and $\lambda_{2}(p,\Omega)$ on disconnected sets, an meaningful model is given when 
\[
\Omega=\mathcal W_{r_1}\cup \mathcal W_{r_2},\text{ with }r_1,r_2>0\text{ and }\mathcal W_{r_{1}}\cap \mathcal W_{r_{2}}=\emptyset.
\] 
We distinguish two cases.
\begin{description}[leftmargin=*]
\item[Case $r_{1}<r_{2}$.] We have
\[
\lambda_{1}(p,\Omega)=\lambda_{1}(p,\mathcal W_{r_{2}}).
\] 
Hence $\lambda_{1}(p,\Omega)$ is simple, and any eigenfunction is identically zero on $\mathcal W_{1}$ and has constant sign in $\mathcal W_{2}$. Moreover,
\[
\lambda_{2}(p,\Omega)=\min\{\lambda_{1}(p,\mathcal W_{r_{1}}),\lambda_{2}(p,\mathcal W_{r_2})\}.
\]
Hence, if $r_{1}$ is not too small, then the second eigenvalue is $\lambda_{1}(\mathcal W_{r_{1}})$, and the second eigenfunctions of $\Omega$ coincide with the first eigenfunctions of $\mathcal W_{r_{1}}$, that do not change sign in $\mathcal W_{r_{1}}$, and vanish on $\mathcal W_{r_{2}}$.
\item[Case $r_{1}=r_{2}$.] We have
\[
\lambda_{1}(p,\Omega)=\lambda_{1}(p,\mathcal W_{r_{i}}),\quad i=1,2.
\]
The first eigenvalue $\lambda_{1}(p,\Omega)$ is not simple: choosing, for example, the function  $U=u_{1}\chi_{\mathcal W_{r_{1}}}-u_{2}\chi_{\mathcal W_{r_{2}}}$, where $u_{i}$, $i=1,2$, is the first normalized eigenfunction of $\lambda_{1}(p,\mathcal W_{r_{i}})$, and $V=u_{1}\chi_{\mathcal W_{r_{1}}}$, then $U$ and $V$ are two nonproportional eigenfunctions relative to $\lambda_{1}(p,\Omega)$. Hence, in this case, by definition,
\[
\lambda_{2}(p,\Omega)=\lambda_{1}(p,\Omega)=\lambda_{1}\left(p,\mathcal W_{r_{i}}\right).
\]
\end{description}
\end{rem}

In order to prove the Hong-Krahn-Szego inequality, we need the following key lemma.
\begin{prop}
\label{keylemma}
Let $\Omega$ be an open bounded set of $\R^{n}$. Then there exists two disjoint domains  $\Omega_1,\Omega_2$ of $\Omega$ such that 
\begin{equation*}
\lambda_{2}(p,\Omega)=\max\{\lambda_{1}(p,\Omega_1),\lambda_{1}(p,\Omega_{2})\}.
\end{equation*}
\end{prop}
\begin{proof}
Let $u_2\in W^{1,p}_0(\Omega)$ be a second normalized eigenfunction. First of all, suppose that $u_{2}$ changes sign in $\Omega$. Then, consider two nodal domains $\Omega_{1}\subseteq \Omega_{+}$ and $\Omega_{2}\subseteq\Omega_{-}$. By definition, $\Omega_{1}$ and $\Omega_{2}$ are connected sets. The restriction of $u_2$ to $\Omega_1$ is, by Theorem \ref{th_prop2}, a first eigenfunction for $\Omega_{1}$ and hence $\lambda_{2}(p, \Omega)=\lambda_1(p,\Omega_1)$. Analogously for $\Omega_2$, hence
\[
\lambda_{2}(p, \Omega)=\lambda_1(p,\Omega_1)=\lambda_{1}(p,\Omega_{2}),
\]
and the proof of the proposition is completed, in the case $u_{2}$ changes sign. 

In the case that $u_{2}$ has constant sign in $\Omega$, for example $u_{2}\ge 0$, then by Theorem \ref{th_prop2} $\Omega$ must be disconnected. If $\lambda_{1}(p,\Omega)$ is simple, by definition $\lambda_{2}(p,\Omega)>\lambda_{1}(p,\Omega)$. Otherwise, $\lambda_{1}(p,\Omega)=\lambda_{2}(p,\Omega)$.
Hence in both cases, we can consider a first nonnegative normalized eigenfunction $u_{1}$ not proportional to $u_{2}$. 

Observe that in any connected component of $\Omega$, by the Harnack inequality, $u_{i}$, $i=1,2$, must be positive or identically zero. Hence we can choose two disjoint connected open sets $\Omega_{1}$ and $\Omega_{2}$, contained respectively in $\{x\in\Omega\colon u_{1}(x)>0\}$ and  $\{x\in\Omega\colon u_{2}(x)>0\}$. Then, $u_{1}$ and $u_{2}$ are first Dirichlet eigenfunctions in $\Omega_{1}$ and $\Omega_{2}$, respectively, and
\[
\lambda_{1}(p,\Omega)=\lambda_1(p,\Omega_1)\le \lambda_{2}(p,\Omega),\qquad \lambda_{2}(p, \Omega)=\lambda_{1}(p,\Omega_{2}),
\]
and the proof is completed.
\end{proof}

Now we are in position to prove the Hong-Krahn-Szego inequality for $\lambda_{2}(p,\Omega)$.
 \begin{thm}
 \label{hkstheo}
 Let $\Omega$ be a bounded open set of $\R^{n}$. Then
\begin{equation}
\label{HKS}
\lambda_2(p,\Omega)\geq  \lambda_2(p,\widetilde{\mathcal W}),
\end{equation}
where $\widetilde{\mathcal W}$ is the union of two disjoint Wulff shapes, each one of measure $\frac{|\Omega|}{2}$. Moreover equality sign in \eqref{HKS} occurs  if  $\Omega$ is the disjoint union of two  Wulff shapes of the same measure.
\end{thm}

\begin{proof}
Let $\Omega_{1}$ and $\Omega_{2}$ given by Proposition \ref{keylemma}.
By the Faber-Krahn inequality we have
\begin{equation*}
\lambda_2(p,\Omega)= 
\max\{ \lambda_{1}(p,\Omega_1) ,\lambda_1(p,\Omega_2)\}
\ge \max\{\lambda_{1}(p,\mathcal W_{r_1}), \lambda_1(p,\mathcal W_{r_2})\}\end{equation*}
with $|\mathcal W_{r_i}|=|\Omega_{i}|$. By the rescaling property of $\lambda_{1}(p,\,\cdot\,)$, and observing that, being $\Omega_{1}$ and $\Omega_{2}$ disjoint subsets of $\Omega$, $| \Omega_{1}|+| \Omega_{2}|\le |\Omega|$, we have that
\begin{multline*}
\max\left\{\lambda_{1}(p,\mathcal W_{r_1}), \lambda_1(p,\mathcal W_{r_2})\right\}
= \lambda_1(p,\mathcal W)\kappa_n^{\frac pn}\max\left\{|\Omega_1|^{-\frac pn}, |\Omega_2|^{-\frac pn}\right\}\ge \\ \ge
\lambda_1(p,\mathcal W)\kappa_n^{\frac pn}\left(\frac{|\Omega|}{2}\right)^{-\frac pn}=\lambda_{1}(p, \widetilde{\mathcal W})
.
\end{multline*}
\end{proof}
  
\section{The limit case \texorpdfstring{$p\to \infty$}{TEXT}}
In this section we derive some information on $\lambda_{2}(p,\Omega)$ as $p$ goes to infinity.  
First of all we recall some known result about the limit of the first eigenvalue.
Let us consider a bounded open set $\Omega$. 

The anisotropic distance of $x\in\overline\Omega$ to the boundary of $\Omega$ is the function 
\begin{equation*}
d_{F}(x)= \inf_{y\in \de \Omega} F^o(x-y), \quad x\in \overline\Omega.
\end{equation*}

We stress that when $F=|\cdot|$ then $d_F=d_{\mathcal{E}}$, the Euclidean distance function from the boundary.

It is not difficult to prove that $d_{F}$ is a uniform Lipschitz function in $\overline \Omega$ and
\begin{equation*}
  F(\nabla d_F(x))=1 \quad\text{a.e. in }\Omega.
\end{equation*}
Obviously, $d_F\in W_{0}^{1,\infty}(\Omega)$. Let us consider the quantity
\begin{equation*}
\rho_{F}=\max \{d_{F}(x),\; x\in\overline\Omega\}.
\end{equation*}
If $\Omega$ is connected, $\rho_{F}$ is called the anisotropic inradius of $\Omega$. If not, $\rho_{F}$ is the maximum of the inradii of the connected components of $\Omega$.  

For further properties of the anisotropic distance function we refer
the reader to \cite{cm07}.
\begin{rem}
It is easy to prove (see also \cite{jlm,bkj}) that the distance function satisfies
\begin{equation}
\label{propminimo}
\frac{1}{\rho_{F}(\Omega)} =\frac{1}{\|d_{F}\|_{L^{\infty}(\Omega)}}=
\min_{\varphi\in W_{0}^{1,\infty}(\Omega)\setminus\{0\}}\frac{\|F(\nabla \varphi)\|_{L^{\infty}(\Omega)}}{\| \varphi\|_{L^{\infty}(\Omega)}}.
\end{equation}
Indeed it is sufficient to observe that if $\varphi\in C_{0}^{1}(\Omega)\cap C(\overline\Omega)$, then $\varphi\in C_{0}^{1}(\Omega_{i})\cap C(\overline\Omega_{i})$, for any connected component $\Omega_{i}$ of $\Omega$. Then for a.e. $x\in \Omega_{i}$, for $y\in \de \Omega_{i}$ which achieves $F^{o}(x-y)=d_{F}(x)$, it holds
\begin{multline*}
|\varphi(x)|=|\varphi(x)-\varphi(y)|=
|\langle\nabla \varphi(\xi), x-y\rangle| \le \\ \le F(\nabla \varphi(\xi)) \,F^{o}(x-y) \le \|F(\nabla \varphi)\|_{L^{\infty}(\Omega)} d_{F}(x).
\end{multline*}
Passing to the supremum and by density we get \eqref{propminimo}. 
\end{rem}

The following result holds (see \cite{bkj,jlm}).
\begin{thm}
\label{limprimo}
Let $\Omega$ be a bounded domain in $\R^{n}$, and let $\lambda_{1}(p,\Omega)$ be the first eigenvalue of \eqref{eigpb}. Then
\[
\lim_{p \to \infty}\lambda_{1}(p,\Omega)^{\frac 1p}=\frac{1}{\rho_{F}(\Omega)}. 
\]
\end{thm}
Now let us define
\begin{equation*}
\Lambda_1(\infty,\Omega)= \frac{1}{\rho_{F}(\Omega)}. 
\end{equation*}
The value $\Lambda_{1}(\infty,\Omega)$ is related to the so-called anisotropic infinity Laplacian operator defined in \cite{bkj}, that is
\begin{equation*}
\mathcal Q_{\infty} u = \langle\nabla^{2} u\, J(\nabla u), J(\nabla u) \rangle,
\end{equation*}
where $J(\xi)=\frac12\nabla_{\xi}\left[F^{2}\right](\xi)$. Note that we mean, by continuous extension, $J(0)=0$. This is possible being $F$ $1$-homogeneous and $F(0)=0$.

Indeed, in \cite{bkj} the following result is proved. 
\begin{thm}
\label{eqfirst}
Let $\Omega$ be a bounded domain in $\R^{n}$. Then, there exists a positive solution $u_{\infty}\in W_{0}^{1,\infty}(\Omega)\cap C(\bar\Omega)$ which satisfies, in the viscosity sense, the following problem:
\begin{equation}
\label{firstinf}
\left\{
\begin{array}{ll}
\min\{F(\nabla u)-\Lambda u,-\mathcal Q_{\infty}u\}=0 &\text{in }\Omega,\\[.2cm]
u=0 &\text{on }\de\Omega.
\end{array}
\right.
\end{equation}
with $\Lambda=\Lambda_{1}(\infty,\Omega)$.  Moreover, any positive solution $v\in W_{0}^{1,\infty}(\Omega)$ to \eqref{firstinf} with $\Lambda=\Lambda_{1}(\infty,\Omega)$ satisfies  
\[
\frac{\|F(\nabla v)\|_{L^{\infty}(\Omega)}}{\|v\|_{L^{\infty}(\Omega)}}
=\min_{\varphi\in W_{0}^{1,\infty}(\Omega)\setminus\{0\}}
\frac{\|F(\nabla \varphi)\|_{L^{\infty}(\Omega)}}{\|\varphi\|_{L^{\infty}(\Omega)}}=\Lambda_1(\infty,\Omega)= \frac{1}{\rho_{F}(\Omega)}. 
\]
Finally, if problem \eqref{firstinf} admits a positive viscosity solution in $\Omega$, then $\Lambda=\Lambda_{1}(\infty,\Omega)$.
\end{thm}
\begin{prop}
\label{propprimo}
Theorem \ref{limprimo} holds also when $\Omega$ is a bounded open set of $\R^{n}$.
\end{prop}
\begin{proof}
Suppose that $\Omega$ is not connected, and consider a connected component $\Omega_{0}$ of $\Omega$ with anisotropic inradius $\rho_{F}(\Omega)$.
By the monotonicity property of $\lambda_{1}(p,\Omega)$ given in  Proposition \ref{fireig}, we have 
\[
\lambda_{1}(p,\Omega) \le \lambda_{1}(p,\Omega_{0}).
\]
Then up to a subsequence,  passing to the limit as $p\to +\infty$  and using Theorem \ref{limprimo} we have
\begin{equation}
\label{eq:sconnessi1}
\tilde \Lambda=\lim_{p_{j} \to \infty} \lambda_{1}(p_{j},\Omega)^{\frac{1}{p_{j}}}\le \frac{1}{\rho_{F}(\Omega)}. 
\end{equation}
In order to prove that $\tilde \Lambda=\rho_{F}(\Omega)^{-1}$, 
let $u_{p_{j}}$ the first nonnegative normalized eigenfunction associated to $\lambda_{1}(p_{j},\Omega)$. Reasoning as in \cite{bkj}, the sequence $u_{p_{j}}$ converges to a function $u_{\infty}$ in $C^{0}(\Omega)$ which is a viscosity solution of \eqref{firstinf} associated to $\tilde \Lambda$. Then by the maximum principle contained in \cite[Lemma 3.2]{bb}, in each connected component of $\Omega$, $u_{\infty}$ is either positive or identically zero. Denoting by $\tilde\Omega$ a connected component of $\{u_{\infty}>0\}$, by the uniform convergence, for $p_{j}$ large, also $u_{p_{j}}$ is positive in $\tilde \Omega$. Then by Theorem \ref{th_prop2} we have
\[
\lambda_{1}(p_{j},\tilde\Omega) = \lambda_{1}(p_{j},\Omega),\quad\text{and then}\quad\frac{1}{\rho_{F}(\tilde\Omega)}=\tilde \Lambda.
\]
By \eqref{eq:sconnessi1} and by definition of $\rho_{F}$, $\tilde \Lambda\le \rho_{F}(\Omega)^{-1}\le\rho_{F}(\tilde \Omega)^{-1}=\tilde\Lambda$; then necessarily $\tilde \Lambda = \rho_{F}(\Omega)^{-1}$.
\end{proof}

In order to define the eigenvalue problem for $\mathcal Q_{\infty}$, let us consider the following operator
\begin{equation*}
\mathcal A_{\Lambda}(s,\xi,X) = 
\left\{
\begin{array}{ll}
\min\{ F(\xi)-\Lambda s, - \langle X\,J(\xi), J(\xi) \rangle \} &\text{ if } s>0,\\[.2cm]
- \langle X\,J(\xi), J(\xi) \rangle &\text{ if } s=0,\\[.2cm]
\max\{ -F(\xi)-\Lambda s, - \langle X\,J(\xi), J(\xi) \rangle \}&\text{ if } s<0,
\end{array}
\right.
\end{equation*}
with $(s,\xi,X) \in \R \times\R^{n}\times S^{n\times n}$, where $S^{n\times n}$ denotes the space of real, symmetric matrices of order $n$.  
Clearly $\mathcal A_{\Lambda}$ is not continuous in $s=0$.

For completeness we recall the definition of viscosity solution for the operator $\mathcal A_{\Lambda}$.
\begin{defn}
\label{defsolvisc}
Let $\Omega \subset \R^{n}$ a bounded open set. A function $u \in C(\Omega)$ is a viscosity subsolution (resp. supersolution) of $\mathcal A_{\Lambda}(x,u,\nabla u)=0$ if 
\[
\mathcal A_{\Lambda}(\phi(x),\nabla \phi(x), \nabla ^{2}\phi(x)) \le 0 \quad (\text{resp. } \mathcal A_{\Lambda}(\phi(x),\nabla \phi(x), \nabla^{2}\phi(x)) \ge 0),
\] 
for every $\phi \in C^{2}(\Omega)$ such that $u-\phi$ has a local maximum (resp. minimum) zero at $x$.  A function $u \in C(\Omega)$ is a viscosity solution of $\mathcal A_{\Lambda}=0$ if it is both a viscosity subsolution and a viscosity supersolution and in this case the number $\Lambda$ is called an eigenvalue for $\mathcal Q_{\infty}$.
\end{defn}
\begin{defn}
\label{defeig}
We say that $u\in C(\bar \Omega)$, $u|_{\de\Omega}=0$, $u\not\equiv 0$ is an eigenfunction for the anisotropic $\infty-$Laplacian if there exists $\Lambda\in\R$ such that
\begin{equation}
\label{viscinfty}
 \mathcal A_{\Lambda}(u,\nabla u,\nabla^{2}u)=0 \quad\text{in }\Omega
\end{equation}
in the viscosity sense. Such value $\Lambda$ will be called an eigenvalue for the anisotropic $\infty-$Laplacian.
\end{defn}

In order to define the second eigenvalue for $\mathcal Q_{\infty}$
we introduce the following number:
\begin{equation*}
\rho_{2,F}(\Omega)=\sup \{ \rho>0 \colon \text{there are two disjoint Wulff shapes }\mathcal W_{1}, \mathcal W_{2}\subset\Omega \text{ of radius }\rho\},
\end{equation*}
and let us define
\[
\Lambda_2(\infty,\Omega)=  \frac{1}{\rho_{2,F}(\Omega)}.
\]

Clearly 
\begin{equation*}
\Lambda_1(\infty,\Omega) \le \Lambda_2(\infty,\Omega).
\end{equation*}

\begin{rem}
It is easy to construct open sets $\Omega$ such that 
$\Lambda_1(\infty,\Omega)=\Lambda_2(\infty,\Omega)$. For example, this holds when $\Omega$ coincides with the union of two disjoint Wulff shapes with same measure, or their convex envelope.
\end{rem} 

\begin{rem}
A simple example of $\rho_{2,F}(\Omega)$ is given when $\Omega$ is the union of two disjoint Wulff sets, $\Omega=\mathcal W_{r_{1}}\cup W_{r_{2}}$, with $r_{2}\le r_{1}$. In this case, $\Lambda_{1}(\infty,\Omega)=\frac{1}{r_{1}}$ and, if $r_{2}$ is not too small, then $\Lambda_{2}(\infty,\Omega)=\frac{1}{r_{2}}$.
\end{rem}
\begin{thm}
\label{limsecondo}
Let $\Omega \subset \R^n$ be a bounded open set and let $\lambda_{2}(p,\Omega)$ be the second Dirichlet eigenvalue of $-\mathcal Q_{p}$ in $\Omega$. Then 
\[
\lim_{p\to \infty}\lambda_{2}(p,\Omega)^{\frac1p} = \Lambda_2(\infty,\Omega)= \frac{1}{\rho_{2,F}(\Omega)}.
\]
Moreover $\Lambda_2(\infty,\Omega)$ is an eigenvalue of $\mathcal Q_{\infty}$,  that is  $\Lambda_2(\infty,\Omega)$ is an eigenvalue for the anisotropic infinity Laplacian in the sense of Definition \ref{defeig}.
\end{thm}

\begin{proof}

First we observe that $\lambda_{2}(p,\Omega)^{\frac1p}$ is bounded from above with respect to $p$. More precisely  we have
\begin{equation}
\label{limsup}
 \Lambda_1(\infty,\Omega)\le\limsup_{p \to \infty}\lambda_{2}(p,\Omega)^{\frac1p} \le \Lambda_2(\infty,\Omega).
\end{equation}
Indeed if we consider  two disjoint Wulff shapes $\mathcal W_{1}$ and  $\mathcal W_{2}$ of radius $\rho_{2,F}(\Omega)$, clearly $\mathcal W_{1}\cup\mathcal W_{2}\subset\Omega$  and then by monotonicity property (Proposition \ref{monotone}) of $\lambda_{2}(p,\Omega)$ we have
\[
\lambda_{1}(p,\Omega)^{\frac1p} \le \lambda_{2}(p,\Omega)^{\frac1p}  \le \lambda_{2}(p,\mathcal W_{1}\cup\mathcal W_{2})^{\frac1p} =\lambda_{1}(p,\mathcal W_{1})^{\frac1p},
\]
where last equality follows from Remark \ref{Wulff}.
Then passing to the limit as $p\to \infty$ in the right hand side, by Theorem \ref{limprimo} we have \eqref{limsup}. Hence there exists a sequence $p_{j}$ such that $p_{j}\to +\infty$ as $j\to \infty$, and
\begin{equation}
\label{lambdaconfr}
 \frac{1}{\rho_{F}(\Omega)}=\Lambda_1(\infty,\Omega)\le\lim_{j \to \infty}\lambda_{2}(p_{j},\Omega)^{\frac1p_{j}}=\overline{\Lambda}\le  \Lambda_2(\infty,\Omega)= \frac{1}{\rho_{2,F}(\Omega)}.
\end{equation}
In order to conclude the proof we have to show that $\overline{\Lambda}$ is an eigenvalue for $\mathcal Q_{\infty}$ and that $\overline{\Lambda}= \Lambda_2(\infty,\Omega)$.

Let us consider $u_{j}\in W_{0}^{1,p}(\Omega)$ eigenfunction of $\lambda_{2}(p_{j},\Omega)$ such that $\|u_{j}\|_{L^{p_{j}}(\Omega)}=1$.
Then by standard arguments $u_{j}$, converges, up to a subsequence of $p_{j}$, uniformly to a function $u\in W_{0}^{1,\infty}(\Omega)\cap C(\bar\Omega)$. The function $u$ is a viscosity solution of \eqref{viscinfty} with $\Lambda=\bar\Lambda$. Indeed, let $x_{0}\in\Omega$. If $u(x_{0})>0$, being $u$ continuous, it is positive in a sufficiently small ball centered at $x_{0}$. Then it is possible to proceed exactly as in \cite{bkj} in order to obtain that, in the viscosity sense,
\[
\min\{F(\nabla u(x_{0}))-\overline \Lambda u(x_{0}), -\mathcal Q_{\infty}u(x_{0})\}=0.
\]
Similarly, if $u(x_{0})<0$ then
\[
\max\{-F(\nabla u(x_{0}))-\overline \Lambda u(x_{0}), -\mathcal Q_{\infty}u(x_{0})\}=0.
\]
It remains to consider the case $u(x_{0})=0$. We will show that $u$ is a subsolution of \eqref{viscinfty}. 

Let $\varphi$ a $C^{2}(\Omega)$ function such that $u-\varphi$ has a strict maximum point at $x_{0}$. By the definition of $\mathcal A_{\bar \Lambda}$, we have to show that $-\mathcal Q_{\infty}\varphi(x_{0})\le 0$.

For any $j$, let $x_{j}$ be a maximum point of $u_{j}-\varphi$, so that $x_{j}\to x_{0}$ as $j\to \infty$. Such sequence exists by the uniform convergence of $u_{j}$. By \cite[Lemma 2.3]{bkj} $u_{j}$ verifies in the viscosity sense $-\mathcal Q_{p}u_{j}=\lambda_{2}(p_{j},\Omega)|u_{j}|^{p_{j}-2}u_{j}$. Then 
\begin{multline*}
-\mathcal Q_{p} \varphi_{j}(x_{j})=\\[.2cm]
=
-(p_{j}-2)F^{p_{j}-4}(\nabla \varphi(x_{j}))\langle \nabla^{2}\varphi(x_{j})\, J(\nabla \varphi(x_{j})),J(\nabla \varphi(x_{j}))\rangle +\\[.2cm]
\hfill- F^{p_{j}-2}(\nabla \varphi(x_{j}))\nabla^{2}\varphi(x_{j}) \otimes \nabla_{\xi}J(\nabla \varphi(x_{j})) 
=\\[.2cm]
=-(p_{j}-2)F^{p_{j}-4}(\nabla \varphi(x_{j}))\mathcal Q_{\infty}\varphi(x_{j})-F^{p_{j}-2}(\nabla \varphi(x_{j}))\mathcal Q_{2}\varphi(x_{j}) \le \\
\le \lambda_{2}(p_{j},\Omega)|u_{j}(x_{j})|^{p_{j}-2}u_{j}(x_{j})
;
\end{multline*}
here $A\otimes B:=\sum_{i,k}A_{ik}B_{ik}$, for two $n\times n$ matrices $A,B$. If $\nabla \varphi(x_{0})\ne 0$, then dividing the above inequality by $(p_{j}-2)F^{p_{j}-4}(\nabla \varphi)$ we have
\[
-\mathcal Q_{\infty}\varphi(x_{j})\le 
\frac{F^{2}(\nabla\varphi(x_{j}))\mathcal Q_{2}\varphi(x_{j})}{p_{j}-2}+\left(\frac{\lambda_{2}(p_{j},\Omega)^{\frac{1}{p_{j}-4}}|u_{j}(x_{j})|}{F(\nabla \varphi (x_{j}))}\right)^{p_{j}-4} \frac{u_{j}(x_{j})^{3}}{p_{j}-2}=:\ell_{j}.
\]
Passing to the limit as $j\to \infty$, recalling that $\varphi \in C^{2}(\Omega)$, $F\in C^{2}(\R^{n}\setminus \{0\})$, $\lambda_{2}(p_{j},\Omega)^{\frac{1}{p_{j}}}\to \bar \Lambda$, $\nabla \varphi(x_{0})\ne 0$ and $u_{j}(x_{j})\to 0$ we get
\[
-\mathcal Q_{\infty}\varphi(x_{0})\le 0.
\]

Finally, we note that if $\nabla \varphi(x_{0})=0$, the above inequality is trivially true. Hence, we can conclude that $u$ is a viscosity subsolution.

The proof that $u$ is also a viscosity supersolution can be done by repeating the same argument than before, considering $-u$.

Last step of the proof of the Theorem consists in showing that $\bar \Lambda=\Lambda_{2}(\infty,\Omega)$. We distinguish two cases.

\noindent\textbf {Case 1}: The function $u$ changes sign in $\Omega$.

Let us consider the following  sets
\begin{equation*}
\Omega^{+}=\{x \in \Omega \colon u(x)>0\} \qquad \Omega^{-}=\{x \in \Omega \colon u(x)<0\}.
\end{equation*}
Being $u \in C^{0}(\Omega)$ then $\Omega^{+}, \Omega^{-}$ are two disjoint open sets of $\R^{n}$  and $|\Omega^{+}|>0$ and $|\Omega^{-}|>0$. 

By Theorem \ref{eqfirst} we have 
\[
\overline\Lambda =\Lambda_1(\infty,\Omega^+)  \quad \text{ and  } \quad \overline\Lambda =\Lambda_1(\infty,\Omega^-).
\] 
Then by definition of $\rho_{2,F}$ we get 
\[
\rho_{F}(\Omega^{+})=\rho_{F}(\Omega^{-})=\frac{1}{\overline \Lambda} \le \rho_{2,F}(\Omega), 
\]
that implies, by \eqref{lambdaconfr} that 
\[
\overline{\Lambda}= \Lambda_{2}(\infty,\Omega).
\]

\noindent \textbf{Case 2}: The function $u$ does not change sign in $\Omega$. 

We first observe that  in this case $\Omega$ cannot be connected. Indeed  since $u_{j}$ converges to $u$ in $C^{0}(\overline\Omega)$,  for sufficiently large $p$ we have that there exist second eigenfunctions relative to $\lambda_{2}(p,\Omega)$  with constant sign in $\Omega$ and this cannot happen if $\Omega$ is connected. 

Then in this case, we have to replace the sequence $u_{j}$ (and then the function $u$) in order to find two disjoint connected open subsets $\Omega_{1},\Omega_{2}$ of $\Omega$, 
such that
\begin{equation}
\label{1csa}
\Lambda_{1}(\infty,\Omega)=\Lambda_{1}(\infty,\Omega_{1})
\end{equation}
and
\begin{equation}
\label{2csa}
\overline{\Lambda}=\Lambda_{1}(\infty,\Omega_{2}).
\end{equation}

Once we prove that such subsets exist, by \eqref{lambdaconfr} and the definition of $\rho_{2,F}$ we obtain
\[
\rho_{F}(\Omega_{2})=\frac{1}{\overline \Lambda}\le \rho_{2,F}(\Omega)\le \rho_{F}(\Omega)=\rho_{F}(\Omega_{1}),
\]
that implies, again by \eqref{lambdaconfr},
\[
\overline{\Lambda}= \Lambda_{2}(\infty,\Omega).
\]

In order to prove \eqref{1csa} and \eqref{2csa}, we consider $u_{1,\infty}$, an eigenfunction associated to $\Lambda_{1}(\infty,\Omega)$, obtained as limit in $C^{0}(\Omega)$ of a sequence $u_{1,p}$ of first normalized eigenfunctions associated to $\lambda_{1}(p,\Omega)$, and consider a connected component of $\Omega$, say $\Omega_{1}$, where $u_{1,\infty}>0$ and such that $\Lambda_{1}(\infty,\Omega)=\Lambda_{1}(\infty,\Omega_{1})$. The argument of the proof of Proposition \ref{propprimo} gives that such $u_{1,\infty}$ and $\Omega_{1}$ exist. Then, let $u_{2,p}\ge 0$ be a normalized eigenfunction associated to $\lambda_{2}(p,\Omega)$ such that for any $p$ sufficiently large, $\spt(u_{2,p})\cap \Omega_{1}=\emptyset$.

The existence of such a sequence is guaranteed from this three observations:
\begin{itemize}
\item if $u_{2,p}$ changes sign for a divergent sequence of $p$'s, then we come back to the case 1;
\item by the maximum principle, in each connected component of $\Omega$ $u_{2,p}$ is either positive or identically zero;
\item the condition $\spt(u_{2,p})\cap \Omega_{1}=\emptyset$ depends from the fact that $u_{2,p}$ can be chosen not proportional to $u_{1,p}$.
\end{itemize}
Hence, there exists $\Omega_{2}$ connected component of $\Omega$ disjoint from $\Omega_{1}$, such that $u_{2,p}$ converges to $u_{2,\infty}$ (up to a subsequence) in $C^{0}(\Omega_{2})$, and where $u_{2,\infty}>0$. By Theorem \ref{eqfirst}, \eqref{2csa} holds.
\end{proof}

\begin{thm}
Given $\Omega$ bounded open set of $\R^{n}$, let $ \Lambda >\Lambda_{1}(\infty,\Omega)$ be an eigenvalue for $\mathcal Q_{\infty}$. Then  $\Lambda \ge \Lambda_{2}(\infty,\Omega)$ and $\Lambda_2(\infty,\Omega)$ is the second eigenvalue of $\mathcal Q_{\infty}$, in the sense that there are no eigenvalues of $\mathcal Q_{\infty}$ between $\Lambda_1(\infty,\Omega)$ and $\Lambda_2(\infty,\Omega)$.
\end{thm}

\begin{proof}
Let $u_{\Lambda}$ be an eigenfunction corresponding to $\Lambda$. We distinguish two cases.

\noindent\textbf {Case 1}: The function $u_{\Lambda}$ changes sign in $\Omega$.

Let us consider the following  sets
\begin{equation*}
\Omega^{+}=\{x \in \Omega \colon u_{\Lambda}(x)>0\} \qquad \Omega^{-}=\{x \in \Omega \colon u_{\Lambda}(x)<0\}.
\end{equation*}
Being $u_{\Lambda} \in C^{0}(\Omega)$ then $\Omega^{+}, \Omega^{-}$ are two disjoint open sets of $\R^{n}$  and $|\Omega^{+}|>0$ and $|\Omega^{-}|>0$. 

By Theorem \ref{eqfirst} we have 
\[
\Lambda =\Lambda_1(\infty,\Omega^+)  \quad \text{ and  } \quad \Lambda =\Lambda_1(\infty,\Omega^-).
\] 
Then by definition of $\rho_{2,F}$ we get 
\[
\rho_{F}(\Omega^{+})=\rho_{F}(\Omega^{-})=\frac{1}{\Lambda} \le \rho_{2,F}(\Omega), 
\]
that implies, by \eqref{lambdaconfr} that 
\[
\Lambda\ge \Lambda_{2}(\infty,\Omega).
\]

\noindent \textbf{Case 2}: The function $u_{\Lambda}$ does not change sign in $\Omega$.

By Theorem \ref{eqfirst} $\Omega$ cannot be connected being $ \Lambda >\Lambda_{1}(\infty,\Omega)$.

In this case, again by By Theorem \ref{eqfirst} we can find two disjoint connected open subsets $\Omega_{1},\Omega_{2}$ of $\Omega$, 
such that
\begin{equation}
\label{1cs}
\Lambda_{1}(\infty,\Omega)=\Lambda_{1}(\infty,\Omega_{1})
\end{equation}
and
\begin{equation}
\label{2cs}
\Lambda=\Lambda_{1}(\infty,\Omega_{2}).
\end{equation}

Being $ \Lambda >\Lambda_{1}(\infty,\Omega),$   we obtain
\[
\rho_{F}(\Omega_{2})=\frac{1}{\Lambda}<\rho_{F}(\Omega)=\rho_{F}(\Omega_{1}),
\]
that by the definition of $\rho_{2,F}$ implies,
\[
\Lambda\ge \Lambda_{2}(\infty,\Omega).
\]

\end{proof}

\begin{rem}
We observe that if $\Omega$ is a bounded open set and $\widetilde{\mathcal  W}$ is the union of two disjoint Wulff sets with the same measure $|\Omega|/2$, it holds that
\[
\rho_{2,F}(\Omega) \le \rho_{2,F}(\widetilde{\mathcal  W}),
\]
that is, 
\[
\Lambda_{2}(\infty,\Omega)\ge \Lambda_{2}(\infty,\widetilde{\mathcal W}),
\]
that is the Hong-Krahn-Szego inequality for the second eigenvalue of $-\mathcal Q_{\infty}$.
\end{rem}

\section*{Acnowledgements}

This work has been partially supported by the FIRB 2013 project ``Geometrical and qualitative aspects of PDE's'' and by GNAMPA of INdAM.

\small{

}

\end{document}